\documentclass[12pt]{article}
\usepackage{amssymb}
\usepackage{graphicx}
\usepackage{amsmath, amsthm}
\usepackage{hyperref}
\usepackage{subfigure} 

\renewcommand{\baselinestretch}{2}

\newtheorem{theorem}{Theorem}

\newtheorem{lem}{Lemma}
\newtheorem{cor} {Corrolary}

\begin{document}
\renewcommand{\baselinestretch}{1}
\title{Stationarity and ergodicity of vector STAR models}
\author{\textsc{Igor L. Kheifets$^{a}$} \and \textsc{Pentti J.
Saikkonen$^{b}$}
\vspace{0.008in} \and $^{a}$ITAM, Mexico City \and $^{b}$Department of
Mathematics and Statistics, University of Helsinki
}
\maketitle

\renewcommand{\baselinestretch}{1.2}
\begin{abstract} Smooth transition autoregressive models are widely used to
capture nonlinearities in univariate and multivariate time series. Existence of
stationary solution is typically assumed, implicitly or explicitly.  In this
paper we describe conditions for stationarity and ergodicity of vector STAR
models. The key condition is that the joint spectral radius of certain matrices
is below $1$. It is not sufficient to assume that  separate spectral radii are below
$1$. Our result allows to use recently introduced toolboxes from computational
mathematics to verify the stationarity and ergodicity of vector STAR models.


\end{abstract}

% ----------------------------------------------------------------

\renewcommand{\baselinestretch}{2}
\section{Introduction}
Consider the vector smooth transition autoregressive model (see Hubrich and
Terasvirta, 2013)
\begin{align}\label{eq:y}
y_t=
\mu_0 + \sum_{j=1}^p \Phi_j y_{t-j} 
+ G(\gamma,c;s_t) \left\{
\mu_1 + \sum_{j=1}^p \Psi_j y_{t-j} 
\right\}
+\varepsilon_t,
\end{align}
where  $y_t$ and $\varepsilon_t$ are $n\times 1$ random vectors, $\mu_0$ and
$\mu_1$ are $n\times 1$ intercept vectors, $\Phi_j$ and $\Psi_j$,
$j=1,\ldots,p$, are $n\times n$ parameter matrices.  We assume that the random
vectors $\varepsilon_t$ are i.i.d., with zero mean and any positive definite
covariance matrix and with a density bounded away from zero on compact subsets
of $R^n$.

Furthermore, the continuous function of random variable  $s_t$ and parameters
$\gamma>0$ and $c$,
$G(\gamma,c;s_t)$ takes values on $[0,1]$ and is called the transition
function.  Random variable $s_t$ is a function of $\{y_{t-j}, j=1,\ldots,p\}$.
For example, it could be a logistic function
\begin{equation}
G(\gamma,c;s_t)=\left(1+\exp(-\gamma(s_t-c))\right)^{-1},
\end{equation}
and $s_t=y_{t-j,i}$ for some $j\in\{1,\ldots,p\}$ and $i\in\{1,\ldots,n\}$.

The goal of this paper is to formulate conditions for the existence of a stationary
solution to (\ref{eq:y}).  We don't aim to cover the most general case;
instead, we provide an explicit treatment of the most popular models at the same
time trying to keep exposition simple. We start with a simple model with two
regimes and homoskedastic errors. Then we extend it to a more general situation
with several regimes and regime-dependent error covariance matrix. 

Conditions for stationarity  exist for regime switching vector error correction models, see 
Bec and Rahbek (2004), Saikkonen (2005, 2008). We are not aware of
corresponding conditions for vector STAR models.

The rest of the paper is organized as follows. Section~2 introduces the concept
of the joint spectral radius of a set of matrices and shows how it can be used.
Section~3 provides a numerical illustration. Section~4 contains our general
result. Then we conclude.

\section{Joint spectral radius}
For the model in Equation~(\ref{eq:y}) define matrices 
\begin{align*}
B_1 =
 \begin{pmatrix}
   \Phi_1 + \Psi_1 & \Phi_2 + \Psi_2 & \cdots & \Phi_{p-1} + \Psi_{p-1} & \Phi_p + \Psi_p \\
     I_n & 0  & \cdots & 0 & 0 \\
     0 & I_n  & \cdots & 0 & 0 \\
       \vdots  & \vdots  & \ddots & \vdots & \vdots \\
         0 & 0 & \cdots & I_n  & 0
          \end{pmatrix}
\end{align*}

\begin{align*}
B_2 =
 \begin{pmatrix}
   \Phi_1  & \Phi_2 & \cdots & \Phi_{p-1} & \Phi_p  \\
     I_n & 0  & \cdots & 0 & 0 \\
     0 & I_n  & \cdots & 0 & 0 \\
       \vdots  & \vdots  & \ddots & \vdots & \vdots \\
         0 & 0 & \cdots & I_n  & 0
          \end{pmatrix}
\end{align*}

The \emph{joint spectral radius} (JSR) of a finite set of square matrices $\mathcal{A}$ is defined by
\begin{equation}
\rho(\mathcal{A})=
lim\sup_{j\to\infty}\left(\sup_{A\in \mathcal{A}^j} \rho(A)\right)^{1/j},
\end{equation}
where $\mathcal{A}^j=\{A_1 A_2\ldots A_j:A_i\in\mathcal{A}, i=1,\ldots,j\}$
and $\rho(A)$ is the spectral radius of the matrix $A$, i.e.\ its largest
absolute eigenvalue. See Jungers (2009) for the survey on the JSR.  

\textbf{Assumption R$'$}
The joint spectral radius of matrices $B_1$ and $B_2$ is less than $1$, i.e.\
$\rho(\{B_1,B_2\})<1$.

%\begin{prop}
Suppose that Assumption R$'$ is satisfied. Then there exists a solution to~(\ref{eq:y}), which  is 1) strictly stationary 2) second-order stationary 3)
$\beta$-mixing with geometrically decaying mixing numbers.
%\end{prop}
%\begin{proof}
This statement is a special case of Theorem~\ref{thm:main} below. Indeed,
Equation~(\ref{eq:y}) is a particular case of Equation~(\ref{eq:hety}) in
Section~\ref{sec:main},
setting in the latter $g=1$ and $\Omega_0=\Omega_1$.
%\end{proof}

There are a number of methods to approximate the JSR and verify Assumption {R$'$}, see
Vankeerberghen, Hendrickx and Jungers (2014). For example, Gripenberg (1997) and
Blondel and Nesterov (2005)
describe algorithms to find an arbitrary small interval containing the JSR. 
The procedure of Blondel and Nesterov (2005) provides approximations to
$\rho(\mathcal{A})$  of relative
accuracy $1-\epsilon$ in time polynomial in $dim(B_1)^{(\ln card(\mathcal{A}))/\epsilon}$, where
$dim(B_1)$ is the size of matrix $B_1$ and 
$card(\mathcal{A})$ is the number of matrices in set $\mathcal{A}$, which are
small in a typical application in econometrics involving nonlinear dynamics.  For the model in
Equation~(\ref{eq:y}) computation is polynomial in $(np)^{(\ln 2)/\epsilon}$.    
In the next section we consider a model for the UK macroeconomic variables and
approximate the JSR using a simple method which works out of the box. For that
model, the computation is polynomial in $6^{(\ln 2)/\epsilon}$ and takes a few
seconds.  

In case of large matrices, the following
simple but useful lemmas help to reduce dimensions of matrices and bound the
JSR from below. To keep the exposition simple, we state them for sets consisting
of two matrices, although they hold for any arbitrary (finite) number of matrices.  Their proofs can be found in Protasov (1996), Blondel and Nesterov (2005) and
Jungers (2009, Section 1.2.2).

\begin{lem}[Invariance to convex hull]
For all $\lambda\in[0,1]$, $\rho(\lambda B_1 + (1-\lambda)
B_2)\le\rho(\{B_1,B_2\})$.  
\end{lem}
In particular, $\max(\rho(B_1),\rho(B_2))\le\rho(\{B_1,B_2\})$, which gives a
lower bound for the JSR which is easy to calculate. Also,
\begin{cor}[Necessary condition]
A
necessary condition for Assumption R$'$ is that all eigenvalues must be less than $1$ in absolute value.
\end{cor}
This condition is not sufficient for Assumption R$'$.
Liebscher (2005), page 676, provides a simple example of a process, for which
all eigenvalues are less than $1$ in absolute value, but the JSR is greater than $1$
and, moreover, the process is not ergodic.

\begin{lem}[Nonnegative entries]
For matrices with nonnegative entries the joint spectral radius satisfies
$\rho(B_1 +B_2)/2\le\rho(\{B_1,B_2\})\le\rho(\{B_1+B_2\})$.  
\end{lem}

\begin{lem}[Invariance under linear bijections]
For any invertible matrix $T$,
$\rho(\{B_1,B_2\})=\rho(\{TB_1 T^{-1},TB_2 T^{-1}\})$.  
\end{lem}

Invariance under linear bijections is useful because transformations help to
reduce the dimension of the problem.
If the condition of the following lemma is satisfied, the set of matrices is
called \emph{reducible} and its JSR can be calculated from the JSR of smaller
matrices.
\begin{lem}[Reducibility]
If there exists an invertible matrix $T$ and square matrices of $B_{1,1}$ and
$B_{2,1}$ of equal dimensions, such that
\begin{align*} TB_1 T^{-1} =
 \begin{pmatrix}
     B_{1,1} & B_{1,0}   \\
     0 & B_{1,2}   \\
          \end{pmatrix}
\quad \text{and} \quad TB_2 T^{-1} =
 \begin{pmatrix}
     B_{2,1} & B_{2,0}   \\
     0 & B_{2,2}   \\
          \end{pmatrix},
\end{align*}
then 
$\rho(\{B_1,B_2\})=\max\rho(\{B_{1,1},B_{2,1}\}) ,\rho(\{B_{1,2},B_{2,2}\})$.  
\end{lem}

\section{Numerical Illustration}\label{sec:ex}

A bivariate LSTAR model for joint movement of output growth  ($y_{t1}$) and the
interest rate spread  ($y_{t2}$) in the UK is suggested by Anderson,
Athanasopoulos, and Vahid (2007).  In particular, the conditional mean of
$y_t$, $\mu_t=\left(\mu_{t1},\mu_{t2}\right)'$,  as a function of unknown
parameters $b=\left(b_1,\ldots,b_{11}\right)'$ is
\begin{eqnarray*}
\mu_{t1} &=&b_1+b_2 y_{t-2,1} + b_3 y_{t-3,1} + b_4 y_{t-1,2} +w_t \left(-b_5 y_{t-1,2}\right), \\
\mu_{t2} &=& b_8+b_9 y_{t-3,1} + b_{10} y_{t-1,2}+ b_{11} y_{t-2,2}, \\
\text{where } w_{t} &=& \left(1+\exp\left(-b_6(y_{t-2,1} - b_7)\right)\right)^{-1}.
\end{eqnarray*}%
The logistic smooth transition autoregressive specification for output growth incorporates different regimes and smooth transitions between them.

Anderson, Athanasopoulos, and Vahid (2007) estimated the model by Maximum
Likelihood (ML), while Kheifets (2018) propose tests of the
following null hypothesis
\begin{equation}\label{eq:H0NAR}
H_{0}:y_{t}|\mathcal{I}_t\sim N(\mu_{t} ,\Sigma ),
\end{equation}
where $\Sigma$  is a nonrandom
positive definite covariance matrix and
$\mathcal{I}_t$ is the information set generated by $y_{t-j}, j = 1,2,\ldots$.

Both papers rely on stationarity and ergodicity of the bivariate series.
We assume that $y_t=\mu_t + \varepsilon_t$, where random vectors
$\varepsilon_t$ are i.i.d.\ with zero mean and any positive definite covariance matrix with a density bounded away from zero on
compact subset of $R^n$. Then, in our notation, $\gamma=b_6$, $c= b_7$ and
$s_t=y_{t-2,1}$, and the system has stationary solution if Assumption R$'$  holds
for the following matrices:

\begin{align*}
B_1 =
 \begin{pmatrix}
 0   &b_4-b_5 & b_2   & 0      & b_3  & 0   \\
 0   & b_{10} & 0     & b_{11} & b_9  & 0     \\
 1   & 0      & 0     & 0      & 0    & 0     \\
 0   & 1      & 0     & 0      & 0    & 0     \\
 0   & 0      & 1     & 0      & 0    & 0     \\
 0   & 0      & 0     & 1      & 0    & 0     \\
          \end{pmatrix}
\end{align*}

\begin{align*}
B_2 =
 \begin{pmatrix}
 0   & b_4    & b_2   & 0      & b_3  & 0   \\
 0   & b_{10} & 0     & b_{11} & b_9  & 0     \\
 1   & 0      & 0     & 0      & 0    & 0     \\
 0   & 1      & 0     & 0      & 0    & 0     \\
 0   & 0      & 1     & 0      & 0    & 0     \\
 0   & 0      & 0     & 1      & 0    & 0     \\
          \end{pmatrix}
\end{align*}

The sample used by Anderson, Athanasopoulos, and Vahid (2007) consists of $159$ quarterly time series observations,
dating from 1960:3 to 1999:4, see Figure~\ref{fig:data}. The data is available at
\url{http://qed.econ.queensu.ca/jae/2007-v22.1/anderson-athanasopoulos-vahid/}.
Output growth is calculated as $100$ $\times$ the difference of
logarithms of seasonally adjusted real DGP, and the spread is the
difference between the interest rates on 10 Year Government Bonds and 3 Month
Treasury Bills. 

\begin{figure}[!t]
\centering 
  \includegraphics[width=0.50\textwidth]{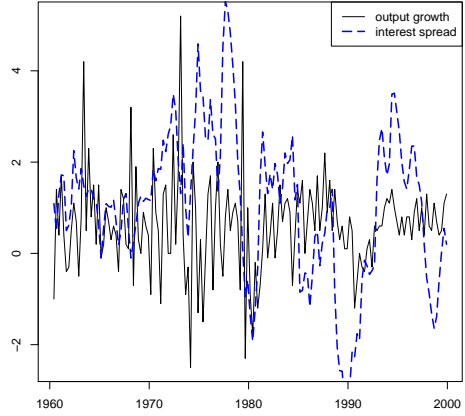}
\caption{UK output growth and interest spread, quarterly data from 1960:3 to
1999:4} \label{fig:data}
\end{figure}

Estimating by ML, we obtain the following coefficients $$\hat b = (0.35, 0.21, 0.15, 0.32, −0.52, 2.56, 0.68,
0.20, −0.14, 1.14, −0.26)′,$$
and standard errors $$
s.e. (\hat b) = (0.09, 0.09, 0.09, 0.11, 0.15, 0.29, 0.12, 0.09, 0.10, 0.10, 0.10)
$$ (calculated using residual parametric bootstrap with $10000$ draws)
and construct corresponding
matrices $\hat B_1$ and $\hat B_2$. 
Imposing stationary initial values in the estimation is not feasible because for
the considered LSTAR model (as well as most other nonlinear (V)AR models) the
stationary distribution is not known. Thus, the ML estimation employed is
necessarily conditional on fixed (and known) initial values. 

The maximum eigenvalues of $\hat B_1$ and $\hat B_2$ are $0.9216$
and $0.8236$, both below $1$, therefore the necessary condition of Assumption R$'$ is satisfied. 
To obtain bounds on the joint spectral radius we use the JSR toolbox written in
Matlab by Raphael Jungers, freely downloadable (with documentation and demos)
from Matlab Central
(\url{https://www.mathworks.com/matlabcentral/fileexchange/33202-the-jsr-toolbox})
The toolbox is described in
Vankeerberghen, Hendrickx and Jungers (2014). In our case, the toolbox provides
the upper and lower bounds which coincide up to the 4th digit and the value is
$0.9216$, calculated in 9 iterations in 31 seconds. Therefore, Assumption R$'$ is
satisfied. 

Notice that the last column of both matrices consists of zeros. That means that
the set $\{B_1,B_2\}$ is reducible to $\{B_{1,1},B_{2,1}\}$ (the other two
matrices are $1$ by $1$ zeros):
\begin{align*}
B_{1,1} =
 \begin{pmatrix}
 0   &b_4-b_5 & b_2   & 0      & b_3   \\
 0   & b_{10} & 0     & b_{11} & b_9     \\
 1   & 0      & 0     & 0      & 0       \\
 0   & 1      & 0     & 0      & 0       \\
 0   & 0      & 1     & 0      & 0       \\
          \end{pmatrix}
\end{align*}

\begin{align*}
B_{2,1} =
 \begin{pmatrix}
 0   & b_4    & b_2   & 0      & b_3   \\
 0   & b_{10} & 0     & b_{11} & b_9     \\
 1   & 0      & 0     & 0      & 0       \\
 0   & 1      & 0     & 0      & 0       \\
 0   & 0      & 1     & 0      & 0       \\
          \end{pmatrix}
\end{align*}
so $\rho\{B_1,B_2\}=\rho\{B_{1,1},B_{2,1}\}$ by Lemma 3 and 4. The JSR toolbox performs this
reduction.

\section{Main Result}\label{sec:main}
We now state our general result for a model with $g+1$ regimes and
heteroskedasticic errors,
\begin{align}\label{eq:hety}
\begin{split}
y_t=&
\mu_0 + \sum_{j=1}^p \Phi_j y_{t-j}+\sum_{i=1}^g G_i(\gamma,c;s_t) \left\{
\mu_{i} + \sum_{j=1}^p \Psi_{i,j} y_{t-j} \right\}\\
&+\left\{\left(1-\sum_{i=1}^g G_i(\gamma,c;s_t)\right)\Omega_{0}
+\sum_{i=1}^g G_i(\gamma,c;s_t)\Omega_i\right\}^{1/2}
\varepsilon_t,
\end{split}
\end{align}
where  $y_t$ and $\varepsilon_t$ are $n\times 1$ random vectors, $\mu_0$ and
$\mu_{i}$ are $n\times 1$ intercept vectors, $\Phi_j$ and $\Psi_{i,j}$,
$j=1,\ldots,p$, are $n\times n$ parameter
matrices, and $\Omega_i$ is positive definite and $i=1,\ldots,g$.  We assume
that random vectors $\varepsilon_t$ are i.i.d.$(0,I_n)$ with a density bounded
away from zero on compact subsets of $R^n$.
If regime switches are only in conditional means,
i.e.~variances coincide in all regimes, $\Omega_0=\Omega_i$ for all
$i=1,\ldots,g$, then the last term is simplified to $\Omega_0^{1/2}\varepsilon_t$.
The transition function is common for all components of vector $y_t$, and it is a
continuous function of random variable  $s_t$ and parameters $\gamma$ and $c$
and takes values on $[0,1]$. Moreover, $G_1 + ... + G_g  \le  1$
and $s_t$ is a function of $\{y_{t-j}, j=1,\ldots,p\}$ as before.

For $i=1,\ldots,g$ define matrices 
\begin{align*}
B_{i} =
 \begin{pmatrix}
   \Phi_1 + \Psi_{i,1} & \Phi_{2} + \Psi_{i,2} & \cdots & \Phi_{p-1} + \Psi_{i,p-1} & \Phi_p + \Psi_p \\
     I_n & 0  & \cdots & 0 & 0 \\
     0 & I_n  & \cdots & 0 & 0 \\
       \vdots  & \vdots  & \ddots & \vdots & \vdots \\
         0 & 0 & \cdots & I_n  & 0
          \end{pmatrix}
\end{align*}
\begin{align*}
B_{g+1} =
 \begin{pmatrix}
   \Phi_1  & \Phi_2 & \cdots & \Phi_{p-1} & \Phi_p  \\
     I_n & 0  & \cdots & 0 & 0 \\
     0 & I_n  & \cdots & 0 & 0 \\
       \vdots  & \vdots  & \ddots & \vdots & \vdots \\
         0 & 0 & \cdots & I_n  & 0
          \end{pmatrix}
\end{align*}

\textbf{Assumption R}
The joint spectral radius of the matrices defined above is less than $1$, i.e.\
$\rho(\{B_1,\ldots,B_g, B_{g+1}\})<1$.

We will use the Markov chain theory on which Mayen and Tweedie (1993) is a standard
reference. In particular, for geometric ergodicity of a Markov chain see Meyn
and Tweedie (1993), Chapter 15.

\begin{theorem}\label{thm:main} Suppose that Assumption R is satisfied.  Then
the process $Z_t=\{y'_t,\ldots,y'_{t-p+1}\}$ is a $(1+\|x\|^2)$-geometrically
ergodic Markov chain.  Therefore, there exist initial values
$y'_{t-p+1},\ldots,y_0'$ such that the process $y_t$ defined in
(\ref{eq:hety}) is 1) strictly stationary 2) second-order stationary 3)
$\beta$-mixing with geometrically decaying mixing numbers.  \end{theorem}

\begin{proof}
Saikkonen (2008) establishes stationarity and mixing properties of nonlinear
vector error correction (VEC) models under Assumption R using the theory of
Markov chains. In order to do so, he transforms the VEC model into a VAR model
which can be formulated as a Markov chain for which
Theorem 15.0.1 of Meyn and Tweedie (1993) can be applied. The main work is to
verify condition (15.3) and related assumptions of that theorem. As our model
is already a VAR model, it is sufficient to show that our vector STAR model can
be written in the form of the nonlinear VAR model in Equation (17) of Saikkonen
(2008) and that the assumptions needed in Theorem 1 of that paper hold true.

First assume that the number of unit
roots $n-r$ in Saikkonen's Assumption 3 is zero so that $n = r$ (in that
assumption $n$ has the same meaning as here). Then note that in the definition
of the matrix $J$ above Saikkonen's Equation (17)  we have (in addition to $n =
r$) $\beta = c = I_n$ and $z_t=y_t$ (see Saikkonen's Note 2). Thus, it follows that $J$ is a
nonsingular matrix of dimension $np$ and we can choose the matrix $S$ in
Equation (17) the inverse of $J'$. In Saikkonen's Equation (17) $h_1 + ... + h_m
= 1$ and we assume that, corresponding to Saikkonen's Assumption 2, the
functions $h_s$ are independent of the argument $\eta_t$. Then, with $m = g+1$
the first term on the right hand side of Equation (17) in Saikkonen (2008) is
\begin{equation}
\sum_{s=1}^{m}h_s\sum_{j=1}^p \bar B_{sj} y_{t-j} 
= \sum_{j=1}^p \bar B_{g+1,j} y_{t-j} 
+ \sum_{i=1}^g h_i\sum_{j=1}^p (\bar B_{ij}-\bar B_{g+1,j}) y_{t-j}.
\end{equation}
Substitute $h_m=1-\sum_{i=1}^g G_i$, $h_{i} = G_i$ and $\bar B_{g+1,j}=\Phi_j$
and $\bar B_{ij} - \bar B_{g+1,j}=\Psi_{i,j}$ to obtain dynamics in Equation (\ref{eq:hety})
without the intercept.  The second term, which is defined in Equation (11) in
Saikkonen (2008),
produces the intercept, because we can take $\rm{I}_2$ as an empty set of indexes.
Finally, the third term gives the required heteroskedastic errors, because
\begin{equation}
\sum_{s=1}^{m}h_s \Omega_{s} 
= \left(1-\sum_{=1}^g G_i\right) \Omega_{0} 
+ \sum_{i=1}^g G_i \Omega_{i}.
\end{equation}

From the preceding discussion and the fact that the random variable $s_t$ is a
function of $\{y_{t-j}, j=1,\ldots,p\}$ we can conclude that
Equation~(\ref{eq:hety}) is a special case of
Equation (17) of Saikkonen (2008). Therefore
Markov chain theory for the process $Z_t=\{y'_t,\ldots,y'_{t-p+1}\}$ can be
applied. The conditions assumed for the error term and transition functions
below Equation~(\ref{eq:hety}) imply that the conditions in Assumptions 1 and
2(a) of Saikkonen (2008) are satisfied whereas condition (b) in his Assumption
2 is dispensable. To see that Theorem 1 of Saikkonen (2008) implies the 
assertions stated in our theorem it now suffices to note that our Assumption R
corresponds to Saikkonen's (2008) condition (19) and his Equation (10) is
dispensable in our case.  \end{proof}

\section{Final Remarks}

In this paper we describe conditions for stationarity and ergodicity of
vector STAR models. The sufficient condition is that the joint spectral radius of
certain matrices is below $1$. This condition can be checked using recently
introduced toolboxes from computational mathematics.

For linear models, for example for the (causal) univariate autoregressive model of order $1$,
it is known that stationarity holds if the autoregressive coefficient
is in $(-1,1)$, otherwise the time series is nonstationary.  For the considered
vector
LSTAR model, the stationarity holds if the joint spectral radius is below $1$.
However, it is only a sufficient condition for stationarity and ergodicity.
Thus, if the joint spectral radius equals one or is larger than one, we can only
conclude that it is not possible to verify stationarity and ergodicity by using
the employed criterion and in principle it is possible that stationarity and
ergodicity still holds.  The conclusion is similar to that discussed above if
the value of the joint spectral radius based on estimates is smaller than one
but deemed to be so close to one that, due to estimation errors, it can be one
or even larger than one with high probability.

An interesting extension of the model considered here would be to add regressors
$x_t$, as in e.g.~Hubrich and Terasvirta  (2013) 
\begin{align}\label{eq:hetyx}
y_t=
\mu_0 + \sum_{j=1}^p \Phi_j y_{t-j} + \Gamma x_t
+ G(\gamma,c;s_t) \left\{
\mu_1 + \sum_{j=1}^p \Psi_j y_{t-j} + \Xi x_t 
\right\}
+\varepsilon_t,
\end{align}
where  $x_t$ are $k\times 1$ random vectors, $\Gamma$ and $\Xi$ are $n\times k$
parameter matrices.
Suppose that $x_t$ is an exogenous random vector. Then no results on stationarity
can hold true if $x_t$ is nonstationary, and even if $x_t$ is assumed stationary
and ergodic Theorem 1 in Saikkonen (2008) is not applicable because without
further assumptions model (8) cannot be cast into the required Markov chain
form.

\section{Acknowledgements}

Igor Kheifets gratefully acknowledges financial support 
from the Spanish Ministerio de Ciencia,
Innovacion y Universidades under Grant ECO2017-86009-P and
thanks the faculty and staff of the New Economic School for their hospitality
during his visits to Moscow.
Pentti Saikkonen thanks the Academy of Finland (grant number 1308628) for
financial support.

\renewcommand{\baselinestretch}{1.2}

\end{document}